\documentclass[11pt]{article}
\usepackage{graphicx}

\usepackage{amssymb,amsfonts,amsthm,amsmath, mathrsfs}
\usepackage{geometry}

  \newcommand{\Z}{\ensuremath{\mathbb{Z}}}%
  \newcommand{\N}{\ensuremath{\mathbb{N}}}%
    \newcommand{\A}{\ensuremath{\mathcal{A}}}%

\theoremstyle{definition}
  \newtheorem{defin}{Definition}[section]
  \newtheorem{definappendix}{Definition}[section]

\theoremstyle{plain}
  \newtheorem{thm}[defin]{Theorem}
  \newtheorem{main thm}{Theorem}
  \newtheorem{prop}[defin]{Proposition}
  
  \newtheorem{cor}[defin]{Corollary}
  \newtheorem{lemma}[defin]{Lemma}
  \newtheorem{lemmaappendix}[definappendix]{Lemma}

\theoremstyle{remark}
  \newtheorem{remark}[defin]{Remark}
  \newtheorem{remarkappendix}[definappendix]{Remark}
  
  \renewcommand\appendix{\par
\gdef\thetable{\Alph{table}}
\gdef\thefigure{\Alph{figure}}
\section*{Appendix: Schreier graphs and Gray code}
\gdef\thesection{\Alph{section}}
\setcounter{section}{1}}

  \author{Nicol\'as Matte Bon\thanks{Universit\'e d'Orsay \& DMA, \'Ecole Normale Sup\'erieure; email: nicolas.matte.bon@ens.fr}}
  
  \title{Topological full groups of minimal subshifts with subgroups of intermediate growth}
  
\begin{document}
  \maketitle
 \abstract{We show that every Grigorchuk group $G_\omega$ embeds in (the commutator subgroup of) the topological full group of a minimal subshift. In particular,  the topological full group of a Cantor minimal system can have subgroups of intermediate growth, a question raised by Grigorchuk;  it can also have finitely generated  infinite torsion subgroups, as well as residually finite subgroups that are  not elementary amenable, answering   questions  of Cornulier. By estimating the word-complexity of this subshift, we deduce that every Grigorchuk group $G_\omega$ can be embedded in a finitely generated simple group that has trivial Poisson boundary for every simple random walk.}
 
 \section{Introduction}
Let $(X, \varphi)$ be a  \emph{Cantor minimal system}, i.e.\ a minimal dynamical system where $X$ is a compact space homeomorphic to the Cantor set, and $\varphi$ is a homeomorphism of $X$. Recall that a dynamical system is said to be minimal if every orbit is dense. The \emph{topological full group} of $(X, \varphi)$ is the group $[[\varphi]]$ of all homeomorphisms of $X$ that locally coincide with a power of $\varphi$; in other words the homeomorphisms $g$ of $X$ so that there exists a continuous function $n:X\to \Z$ that verifies $g(x)=\varphi^{n(x)}(x)$ for every $x\in X$. Giordano, Putnam and Skau show in \cite{Giordano-Putnam-Skau:flipconjugacy} that this group characterizes completely the dynamics of $(X, \varphi)$, up to replacing $\varphi$ by $\varphi^{-1}$.
 
In this note we are mainly interested in the possible behaviors for the growth of finitely generated subgroups of $[[\varphi]]$. The \emph{growth function} of a finitely generated group $G$ equipped with a finite symmetric generating set $S$ is the function $b_{G, S}:\N\to\N$ that counts the number of group elements that can be obtained as a product of at most $n$ generators in $S$. The growth of  $G$ is said  to be \emph{polynomial}  if there exist $C, D>0$ such that $b_{G, S}(n)\leq Cn^D$ for every $n\in\N$, \emph{exponential} if there exists $c>1$ such that $b_{G, S}(n)\geq c^n$ for every $n\in\N$, and \emph{intermediate} otherwise. These properties do not depend on the choice of $S$. 
 
 \emph{Grigorchuk groups} $G_\omega$, first defined and studied by Grigorchuk in \cite{Grigorchuk:grigorchukgroups}, provide examples of groups of intermediate growth. The groups $G_\omega$ can be defined as certain  subgroups of the automorphism group of a binary rooted tree. They are parameterized by an infinite sequence on three symbols, denoted $\omega\in\{\mathbf{0, 1, 2}\}^{\N_+}$. We refer to Section \ref{S: Grigorchuk} for preliminaries on the groups $G_\omega$. 

 \begin{thm}\label{main}
  For every $\omega\in\{\mathbf{ 0, 1, 2}\}^{\N_+}$ there exists a minimal subshift $(X_\omega, \varphi_\omega)$ and an embedding of the Grigorchuk group $G_\omega$ into the topological full group $[[\varphi_\omega]]$. 
  \end{thm}
  
Relying on properties of the groups $G_\omega$ proven in  \cite[Theorem 2.1, Corollary 3.2]{Grigorchuk:grigorchukgroups} we get
\begin{cor} \label{C: main}A finitely generated subgroup of the topological full group of a minimal subshift
  \begin{enumerate}
  \item  can have intermediate growth;
  \item can be an infinite torsion group;
  \item   can be residually finite without being elementary amenable.
\end{enumerate}
  \end{cor}

The question whether a subgroup of the topological full group of a Cantor minimal system can have intermediate growth was originally raised by Grigorchuk (private communication), and  it appears in   Matui \cite{Matui:exponentialgrowth}, where he shows that the group $[[\varphi]]$ and its derived subgroup $[[\varphi]]'$  have exponential growth whenever they are finitely generated; see also  Question (2c) in Cornulier's survey \cite{Cornulier:Bourbaki}.  Points 2 and 3  in Corollary \ref{C: main}  answer to Questions (2c)-(2d) and to part of Question (2b) asked by Cornulier in the same survey \cite{Cornulier:Bourbaki}.\\

The susbshift $(X_\omega,\varphi_\omega)$ in Theorem \ref{main} can be  defined as a space of labeled graph structures on the set of integers, that have the same finite patterns as the orbital Schreier graphs for the action of $G_\omega$ on the boundary of the binary rooted tree. Similar constructions of group actions based on graph colourings have been used, in this context, by  Elek and Monod   \cite{Elek-Monod:freesubgroup} and by van Douwen \cite{Douwen}. 

A minor modification of the construction allows to embed the groups $G_\omega$ in the \emph{commutator subgroup} of the topological full group of a minimal subshift, this is shown in Corollary \ref{C: commutator}. The latter  is a finitely generated simple group by results of Matui \cite[Theorem 4.9, Theorem 5.4]{Matui:simple}. Relying on a result from \cite{simpleliouville} we deduce the following statement in Theorem \ref{P: embedding}: every Grigorchuk group $G_\omega$ can be embedded in a finitely generated simple group with the Liouville property (in particular, amenable). Note that the topological full group of any Cantor minimal system is amenable by a result of Juschenko and Monod \cite{Juschenko-Monod:simpleamenable}. \\
 
Related constructions appear in Vorobets' articles \cite{Vorobets:substitution, Vorobets:Schreier}. In \cite{Vorobets:substitution} he studies dynamical properties of a minimal one-sided substitutional subshift, which can be shown to be a one-sided version of our $(X_\omega, \varphi_\omega)$ for the  periodic sequence $\omega=\mathbf{012\cdots}$ (the group $G_\omega$ corresponding to this $\omega$  is known as the \emph{first Grigorchuk group}). In \cite{Vorobets:Schreier} he studies the dynamics of the action of the first Grigorchuk group on the space of marked Schreier graphs.\\

The paper is structured as follows. Section 2 recalls the definition of Grigorchuk groups and the construction of their orbital Schreier graphs. Proofs concerning this construction are postponed to the Appendix at the end. Section 3 contains the definition of the subshift $(X_\omega, \varphi_\omega)$ and to the proof  of Theorem \ref{main}; this section can be read without reading the Appendix, but it requires notations from Section 2. In Section 4 we modify the construction to get an embedding in the commutator subgroup;  then we study the complexity of the subshift $(X_\omega, \varphi_\omega)$ and deduce the above-mentioned embedding result for Grigorchuk groups. 
\subsubsection*{Acknowledgements} I am grateful to Anna Erschler for several useful discussions, to Slava Grigorchuk for many interesting comments and for pointing out  the connection with Vorobets' work  \cite{Vorobets:substitution, Vorobets:Schreier}, and to Gidi Amir for explaining me the useful relation between the Gray code and  automata groups, first appeared in \cite[Section 10.3]{Bartholdi-Grigorchuk-Sunic:branchgroups}, that is used in the Appendix.
\section{Grigorchuk groups and their Schreier graphs}\label{S: Grigorchuk}
\subsection{Grigorchuk groups}
 
 \label{S: Grigorchuk groups}

Grigorchuk groups act on the infinite binary rooted tree by automorphisms. The binary rooted tree is identified with the set of binary strings $\{0,1\}^*$,  by declaring $v, w\in \{0, 1\}^*$ connected by an edge if there exists $x\in\{0, 1\}$ such that $w=vx$. This tree is rooted at the empty string $\varnothing$.

The groups $G_\omega$ are parameterized by the set $\Omega=\{\mathbf{0, 1, 2}\}^{\N_+}$, which is endowed with a shift map $\sigma:\Omega\to\Omega$ (note that the bold digits $\mathbf{0, 1, 2}$ shall not be confused with the digits $0, 1$ used to index the tree).  Given $\omega\in\Omega$, the Grigorchuk group $G_\omega$ is the group generated by the four automorphisms $a,b_\omega,c_\omega,d_\omega$ of the tree, that we now define. Let $\varepsilon$ be the non-trivial permutation of $\{0,1\}$. The generator $a$ does not depend on the choice of the  sequence $\omega$, and it acts on binary strings by
  \[ a( xv)= \varepsilon(x)v,\] 
  where the above formula holds for every $x\in \{0, 1\}$ and $v\in\{0, 1\}^*$.
   For $i,j\in\{\mathbf{ 0, 1, 2}\}$ let $\varepsilon_{ij}$ be the permutation of $\{0, 1\}$  which is equal to $\varepsilon$ if $i\neq j$ and to the identity if $i=j$. 
 The generators $b_\omega, c_\omega, d_\omega$ are defined  by the recursive rules
\begin{align*}b_\omega (0xv)=0\varepsilon_{\mathbf{2}\omega(1)}(x)v;\qquad b_\omega (1v)= 1b_{\sigma\omega}(v);\\
c_\omega (0xv)=0\varepsilon_{\mathbf{1}\omega(1)}(x)v;\qquad c_\omega (1v)= 1c_{\sigma\omega}(v);\\
 d_\omega (0xv)=0\varepsilon_{\mathbf{0}\omega(1)}(x)v;\qquad d_\omega (1v)= 1d_{\sigma\omega}(v);
  \end{align*}
  that hold for every $x\in\{0, 1\}, v\in \{0, 1\}^*$ and where $\sigma\omega\in\Omega$ is the shifted sequence.
    
Let us give another description of the generator $b_\omega$. If  $v\in\{0,1\}^*$ does not contain the digit $0$, or if the digit 0 appears only at the last position of $v$, then $b_\omega(v)=v$. Otherwise suppose that the first appearance of $0$ in $v$ is at position $j$. If $\omega(j)\neq \mathbf{2}$ then $b_\omega$ permutes the letter at position $j+1$ in $v$, while if $\omega(j)=\mathbf{2}$ then $b_\omega( v)=v$. An analogous description holds for $c_\omega, d_\omega$ after replacing $\mathbf{2}$ by $\mathbf{1, 0}$, respectively.    

Note that for every $\omega\in\Omega$ we have the relations $a^2=b_\omega^2=c_\omega^2=d_\omega^2=b_\omega c_\omega d_\omega=e$.\\

Grigorchuk proved in \cite[Theorem 2.1, Corollary 3.2]{Grigorchuk:grigorchukgroups} that 
\begin{enumerate}
\item the group $G_\omega$ is residually finite and amenable for every sequence $\omega$;
\item if $\omega$ eventually constant the group $G_\omega$ is virtually abelian, otherwise it has intermediate growth and it is not elementary amenable (recall that the class of elementary amenable groups is the smallest class of groups containing finite and abelian groups and which is closed by taking extensions, inductive limits, subgroups and quotients);
\item  the group $G_\omega$ is 2-torsion if $\omega$ has infinitely many appearances of each of the symbols $\mathbf{0, 1, 2}$, otherwise $G_\omega$ admits elements of infinite order. 

\end{enumerate}

 \subsection{Schreier graphs of Grigorchuk groups}\label{S: Schreier}
 The action of Grigorchuk groups on the binary tree extends to the boundary at infinity of the tree, identified with the set of right-infinite strings $\{0, 1\}^\infty$.

Hereinafter let $\rho$ be the constant 1-string $\rho=111\cdots\in \{0, 1\}^\infty$, and let $G_\omega\rho$ be its orbit.  We denote by $\Gamma_\omega$ the \emph{Schreier graph}  for the action of $G_\omega$ on the orbit of $\rho$, with respect to the generating set $S=\{a, b_\omega, c_\omega, d_\omega\}$. This is  the labelled   graph whose vertex set is $G_\omega\rho$ and such that $\gamma, \eta\in   G_\omega\rho$ are connected by an edge if there exists $s\in S$ such that $s\gamma=\eta$. Such an edge is labelled by $s$. Loops and multiple edges are allowed. Since all generators are involutions, the orientation of edges is unimportant and we shall work with non-oriented graphs.

The Schreier graphs $\Gamma_\omega$ can be described elementary, see Bartholdi and Grigorchuk \cite[Section 5]{Bartholdi-Grigorchuk:Hecke} where the graph $\Gamma_\omega$ is constructed for the sequence $\omega=\mathbf{012012}\cdots$ (the corresponding group $G_\omega$ is known as the \emph{first Grigorchuk group}). The construction of the graph $\Gamma_\omega$ for other choices of $\omega$ is also well-known  and can easily be adapted from the case of the first Grigorchuk group. We now recall this construction and some related facts; since we were not able to locate a reference for a generic $\omega$, we refer to the Appendix  for proofs.
\begin{figure}[h]
\centering
\vspace{-20pt}
\includegraphics[scale=.7]{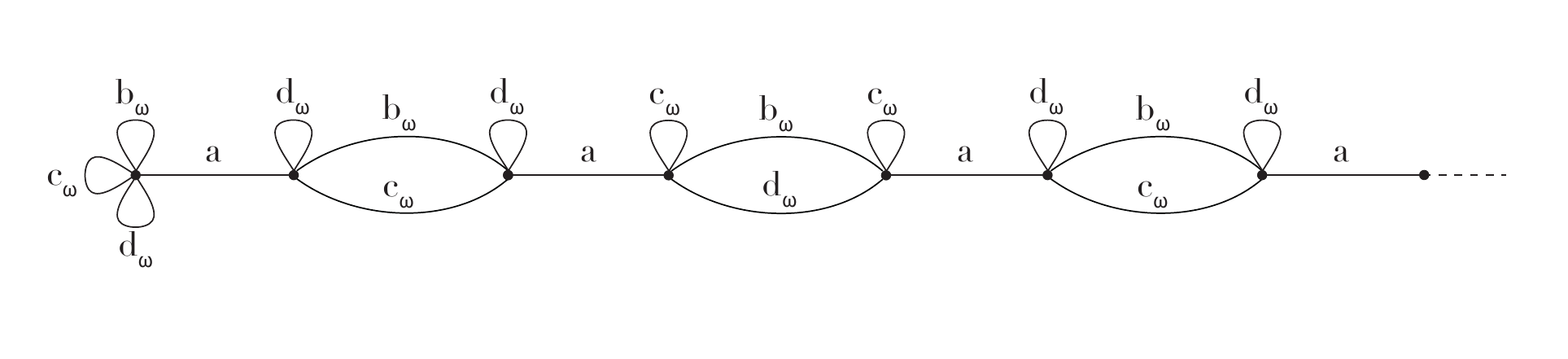}
\vspace{-30pt}
\caption{\label{F: graph} The beginning of the graph $\Gamma_\omega$ for $\omega=\mathbf{012012\cdots}$.}
\end{figure}

\begin{prop} \label{P: orbit} The orbit $G_\omega\rho$ consists exactly of all $\gamma\in\{0, 1\}^\infty$ that are cofinal with $\rho$.
\end{prop}
\begin{prop}\label{P: Schreier}
As an unlabelled graph,  $\Gamma_\omega$ does not depend on the sequence $\omega$ and is isomorphic to the half-line with additional loops and double edges shown in Figure \ref{F: graph} (apart from labels). The endpoint of the half-line is the vertex $\rho$. Moreover for every sequence $\omega$
\begin{itemize}
\item the three loops at $\rho$ are labelled by ``$b_\omega$'', ``$c_\omega$'', ``$d_\omega$'';
\item every simple edge is labelled ``$a$''.

\end{itemize}
The remaining edges (loops and double edges) are labelled by $b_\omega, c_\omega, d_\omega$. The labelling of these depends on $\omega$ and is described in Lemma \ref{P: labelling}.
\end{prop}
Before describing the labelling of $\Gamma_\omega$ we fix some notation.

\begin{defin}\label{D: labelling}\begin{enumerate}\item We denote  $\Theta, \Lambda_\mathbf{0}, \Lambda_\mathbf{1}, \Lambda_\mathbf{2}$  the four finite labelled graphs shown in Figure \ref{F: blocks}, and $\Xi$ the graph with one vertex and three loops, labelled by $b_\omega, c_\omega, d_\omega$. 
\item We denote $*$ the operation of  gluing two graphs by identifying two vertices. For well-definiteness, we  assume that all finite graphs that we consider have a distinguished ``leftmost vertex" and a ``rightmost vertex" (that will be clear from the context); the operation $*$ corresponds to identifying the rightmost  vertex of the first graph with the leftmost vertex of the second.

\end{enumerate}
\end{defin}
\begin{figure}[h]
\centering
\vspace{-20pt}
\includegraphics[scale=.7]{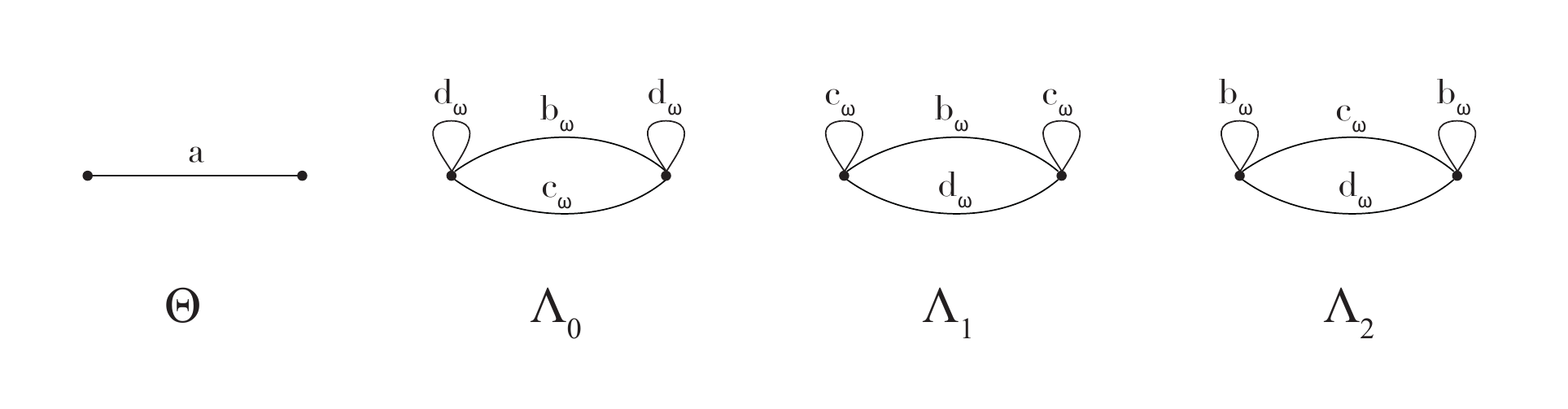}
\vspace{-20pt}
\caption{\label{F: blocks} The ``building blocks'' of $\Gamma_\omega$.}
\end{figure}

By Proposition \ref{P: Schreier}, with this notation  $\Gamma_\omega$  has the form
\[\Gamma_\omega=\Xi*\Theta*\Delta_{1, \omega}*\Theta*\Delta_{2, \omega}*\Theta*\Delta_{3, \omega}*\Theta\cdots, \] 
where $\Delta_{i, \omega}\in\{\Lambda_0, \Lambda_1, \Lambda_2\}$. 

It will be convenient to denote $\tilde{\Gamma}_\omega$  the modification of $\Gamma_\omega$  obtained by erasing the three loops at $\rho$, i.e.
 \[\tilde{\Gamma}_{\omega}=\Theta*\Delta_{1, \omega}*\Theta*\Delta_{2, \omega}*\Theta*\Delta_{3, \omega}\cdots.\]

We now state a lemma, that allows to construct $\tilde{\Gamma}_{\omega}$ recursively (part (i) is sufficient for this purpose). We refer to the Appendix for a proof.
\begin{lemma}\label{P: labelling}\label{L: self-similar}
\begin{itemize}
\item[(i)]For every $n\geq 0$, the the graph spanned by the first  $2^{n+1}$ vertices of $\tilde{\Gamma}_{\omega}$ is isomorphic to the graph $\tilde{\Gamma}_{n, \omega}$, defined recursively  by
\[\left\{\begin{array}{l}\tilde{\Gamma}_{1, \omega}=\Theta*\Lambda_{\omega(1)}*\Theta;\\ \tilde{\Gamma}_{n+1,  \omega}= \tilde{\Gamma}_{n,  \omega}*\Lambda_{\omega(n+1)}*\tilde{\Gamma}_{n,  \omega}, \end{array}\right.\]
where $\omega\in\{\mathbf{0, 1, 2}\}^{\N_+}$ is the sequence that defines $G_\omega$. 

\item[(ii)] The graph $\tilde{\Gamma}_{ \omega}$ has the following self-similarity property: for every $n\geq 1$ we have
\[\tilde{\Gamma}_{\omega}=\tilde{\Gamma}_{n, \omega}*\Delta_{1, \sigma^n\omega}*\tilde{\Gamma}_{n, \omega}*\Delta_{2, \sigma^n\omega}*\tilde{\Gamma}_{n, \omega}*\Delta_{3, \sigma^n\omega}*\tilde{\Gamma}_{n, \omega}*\Delta_{4, \sigma^n\omega}\cdots,\]
where $\sigma^n\omega$ is the  sequence $\omega$ shifted by $n$,  $\tilde{\Gamma}_{ \sigma^n\omega}=\Theta*\Delta_{1,\sigma^n\omega}*\Theta*\Delta_{2, \sigma^n\omega\cdots}$ is the unique decomposition of $\tilde{\Gamma}_{ \sigma^n\omega}$ with $\Delta_{i,\sigma^n\omega}\in \{\Lambda_0, \Lambda_1, \Lambda_2\}$, and $\tilde{\Gamma}_{n,  \omega}$ is the sequence of finite graphs constructed at (i).
\end{itemize}
\end{lemma}

 \section{Construction of the subshift}
 We now turn to the proof of Theorem \ref{main}.

We first address separately the degenerate case of an eventually constant  $\omega$.
 \begin{prop}
Suppose that the sequence $\omega$ is eventually constant. Then for every minimal Cantor system $(X, \varphi)$ there is an embedding of $G_\omega$ into $[[\varphi]]$. 
  \end{prop}
  
 \begin{proof} 
 
If $\omega$ is eventually constant, the group $G_\omega$ is virtually abelian, see \cite[Theorem 2.1 (3)]{Grigorchuk:grigorchukgroups}. In fact it can be embedded in a semi-direct product of the form $\Z^n\rtimes S_n$ for sufficiently large $n$, where the symmetric group $S_n$ acts on $\Z^n$ by permutation of the coordinates (strictly speaking, it is shown in \cite{Grigorchuk:grigorchukgroups} that it can be embedded in $D_\infty^n\rtimes S_n$, where $D_\infty$ is the infinite dihedral group; note for instance that $D_\infty^n\rtimes S_n$ embeds in $\Z^{2n}\rtimes S_{2n}$ by identifying $D_\infty$ with the subgroup of $\Z^2\rtimes S_2$ generated by $(1,-1)\in \Z^2$ and by the nontrivial $\epsilon\in S_2$ ). It is thus sufficient to show that for every minimal Cantor system  $(X, \varphi)$ and every $n\in \N$, the group $\Z^n\rtimes S_n$ embeds in $[[\varphi]]$. Let $U\subset X$ be a clopen set such that $U,\varphi(U),\ldots,\varphi^{n-1}(U)$ are disjoint. For every $0\leq i<j\leq n-1$ let $\sigma_{ij}\in [[\varphi]]$ be the involution with support $\varphi^i(U)\cup\varphi^j(U)$, which coincides with $\varphi^{j-i}$ on $\varphi^{i}(U)$, with $\varphi^{i-j}$ on $\varphi^j(U)$, and is the identity on the complement of $\varphi^i(U)\cup\varphi^j(U)$. The group generated by the involutions $\sigma_{ij}$ is isomorphic to the symmetric group $S_n$. Let $r_0\in[[\varphi]]$ be any element of infinite order with support contained in $U$ (for instance take the \emph{first return map} of $U$ extended to the identity on the complement of $U$). 
For $i=0,\ldots,n-1$ set $r_i=\varphi^ir_0\varphi^{-i}$. The support of $r_i$ is contained in $\varphi^i(U)$, thus the choice of $U$ implies that $r_i$ and $r_j$ commute if $i\neq j$. It follows that the group generated by $r_0,\ldots, r_{n-1}$ is free abelian of rank $n$. The group $\langle\sigma_{ij}, r_l\: :\: 0\leq i< j\leq n-1, \: 0\leq l\leq n-1\rangle$ is then isomorphic to $\Z^n\rtimes S_n$. \qedhere
 \end{proof}
 
{Hereinafter  we fix $\omega\in\Omega$ and assume that it is not eventually constant.}\\

By a \emph{pattern} in a labeled graph we mean the isomorphism class of a finite, connected labelled subgraph. 

\begin{defin}[Construction of the subshift $(X_\omega, \varphi_\omega)$]\begin{enumerate}
\item  Let $X$ be the the space of edge-labeled, connected graphs with vertex set $\Z$ such that 
 any two consecutive integers $n,n+1\in\Z$ are connected by one of the four labelled graphs in Figure \ref{F: blocks}, and there is no edge between $n,m\in\Z$ if $|n-m|>1$. 
 Endow $X$  with the natural product topology and with the shift map $\varphi$ induced by translations of $\Z$, which makes $(X, \varphi)$ conjugate to the full shift over the alphabet $\A=\{\Theta, \Lambda_{\mathbf{0}}, \Lambda_{\mathbf{1}}, \Lambda_\mathbf{2}\}$.

\item Let $X_\omega\subset X$ be the set of all $x\in X$ such that every pattern of $x$ appears in $\tilde{\Gamma}_\omega$ (recall that this is the graph obtained from $\Gamma_\omega$ after erasing the three initial loops).  This defines a  closed, shift-invariant subset of $X$ and we denote $(X_\omega, \varphi_\omega)$ the subshift obtained in this way.
\end{enumerate}

\end{defin} 

\begin{remark}
 Equivalently $X_\omega$ could be defined as the orbit-closure of any $x\in X$ isomorphic to the Schreier graph of the orbit of a point $\gamma\in\{0, 1\}^*$ lying outside the orbit of $\rho$ (these graphs are isomorphic to bi-infinite lines; their labellings admit  similar and only slightly more complicated recursive construction).\end{remark}

A labeled graph is said to be \emph{uniformly recurrent} if for every finite pattern in the graph exists a constant $R>0$ such that every ball of radius $R$ in the graph contains a copy of the pattern.

 \begin{lemma}\label{L: almost periodic}
The graph  $\tilde{\Gamma}_{\omega}$  is uniformly recurrent. Moreover its labelling is not eventually periodic along the half-line.
\end{lemma}
\begin{proof}
We take notations from Definition \ref{D: labelling} and Lemma \ref{P: labelling}.

By part (i) of Lemma \ref{P: labelling}, for every pattern of $\tilde{\Gamma}_{\omega}$ there exists $n$ such that the pattern appears in $\tilde{\Gamma}_{n, \omega}$. Using this, part (ii) of Lemma \ref{L: self-similar}  implies that $\tilde{\Gamma}_{\omega}$ is uniformly recurrent.

Suppose that the labelling of $\tilde{\Gamma}_{\omega}$ is eventually periodic, i.e. that the sequence of graphs $(\Delta_{j,\omega})_{j\geq 1}$ is eventually periodic with period $T$. Since $\tilde{\Gamma}_{1, \omega}=\Theta*\Lambda_{\omega(1)}*\Theta$, part (ii) of Lemma \ref{P: labelling} for $n=1$ implies that all odd terms in the sequence $(\Delta_{j,\omega})_{j\geq 1}$ are equal to $\Lambda_{\omega(1)}$, while the subsequence of even terms is equal to $(\Delta_{j,\sigma\omega})_{j\geq 1}$.
Since we assume that $\omega$ is not eventually constant, part (i) of Lemma \ref{P: labelling} easily implies that the sequence $(\Delta_{j, \omega})_{j\geq 1}$ is also not eventually constant, hence there are infinitely many $i$s for which  $\Delta_{i, \sigma \omega}\neq \Lambda_{\omega(1)}$. We conclude that $T$ is even, and that the sequence$(\Delta_{j,\sigma\omega})_{j\geq 1}$ is also eventually periodic with period $T/2$. Since $\omega$ is not eventually constant, neither is $\sigma\omega$, thus we may repeat the same reasoning for $\tilde{\Gamma}_{\sigma\omega}$ to conclude that $T/2$ is even and $(\Delta_{j,\sigma^2\omega})_{j\geq 1}$ is eventually periodic with period $T/4$. By induction $T$ is a multiple of $2^n$ for every $n$, thus $T=0$.\qedhere
\end{proof}

\begin{cor}\label{L: minimal}
The subshift $(X_\omega, \varphi_\omega)$ is infinite and minimal.
\end{cor}
Previous results by Vorobets in \cite{Vorobets:substitution}  imply minimality of a one-sided subshift given by a certain substitution, which is conjugate to the one-sided version of $(X_\omega, \varphi_\omega)$ for the periodic $\omega=\mathbf{012\cdots}$.
\begin{proof}
Let us first check that $X_\omega$ is not empty.  To see this, let $y\in X$ be any graph that agrees with $\tilde{\Gamma}_{ \omega}$ on the positive half-line. Then any cluster point of $(\varphi^n(y))_{n\geq 0}$  belongs to $X_\omega$.

Let $x\in X_\omega$ be arbitrary. It is a consequence of Lemma \ref{L: almost periodic} that every pattern of $\tilde{\Gamma}_\omega$ appears in $x$ infinitely many times, and that $x$ is uniformly recurrent. Namely every pattern of $\tilde{\Gamma}_\omega$ appears in every sufficiently long segment of $\tilde{\Gamma}_\omega$, thus in every sufficiently long segment of $x$ by the construction of $X_\omega$. This easily implies that the orbit of $x$ is dense in $X_\omega$, and that $(X_\omega, \varphi_\omega)$ is minimal by a well-known characterization of minimality (see \cite{Gottschalk:minimality} or \cite[Proposition 4.7]{Queffelec}). Finally the fact that $\tilde{\Gamma}_\omega$ is not eventually periodic (Lemma \ref{L: almost periodic}) implies that $x$ is not periodic, since $\tilde{\Gamma}_\omega$ and $x$ have the same finite patterns and periodicity can be easily characterized in terms of these (for instance by Morse and Hedlund's  theorem, see \cite[Theorem 4.3.1]{complexity}). Thus $X_\omega$ is infinite.
\end{proof}

\begin{remark}\label{R: every word appears}
It follows from the proof that every pattern in $\tilde{\Gamma}_\omega$ appears in every $x\in X_\omega$.
\end{remark}

\newcommand{\ba}{\bar{a}}
\newcommand{\bc}{\bar{c}_\omega}
\newcommand{\bb}{\bar{b}_\omega}
\newcommand{\bd}{\bar{d}_\omega}
\newcommand{\ob}{b_\omega}
\newcommand{\oc}{c_\omega}
\newcommand{\od}{d_\omega}

\begin{prop}
The Grigorchuk group $G_\omega$ embeds in the topological full group $[[\varphi_\omega]]$.
\end{prop}

\begin{proof}
 Let $x\in X_\omega$. By construction, for every $f\in\{a, \ob, \oc, \od\}$  the vertex  $0\in\Z$ is the endpoint of exactly one edge in $x$  which is labelled by $f$  (this edge is possibly a loop at 0).  We define four elements $\ba, \bb,\bc,\bd$ of the topological full group  $[[\varphi_\omega]]$ as follows: every $\bar{f}\in\{\ba, \bb, \bc, \bd\}$  translates $x\in X_\omega$ in the direction where the label $f$ is, that is $\bar{f}x=\varphi_\omega(x)$ (respectively $\bar{f}x=\varphi^{-1}_\omega(x)$,   $\bar{f}x=x$)  if the label $f$ is on an edge connecting 0 and 1 (respectively on an edge connecting 0 and -1,  on the loop at 0).   It is clear that $\ba, \bb,\bc,\bd\in[[\varphi_\omega]]$.

We claim that the map $\iota:a\mapsto \ba, \ob\mapsto\bb, \oc\mapsto\bc, \od\mapsto\bd$ extends to an injective group homomorphism $\iota: G_\omega\to[[\varphi_\omega]]$.

To see that $\iota$ extends to a well-defined homomorphism, it is sufficient to check that relations are respected, i.e. that for every $n\in\N$ and $h_1,\ldots,h_n\in \{a, \ob, \oc, \od\}$ such that $h_1\cdots h_n=e$ in $G_\omega$ we have $\bar{h}_1\cdots\bar{h}_n=e$ in $[[\varphi_\omega]]$. Suppose that $\bar{h}_1\cdots\bar{h}_n\neq e$ in $[[\varphi_\omega]]$. Then there exists a point $x\in X_\omega$ such that $\bar{h}_1\cdots \bar{h}_nx\neq x$.  Write $x|_n$ for the finite subgraph of $x$ spanned by the interval $[-n,n]\subset \Z$. 
Observe that the fact whether $x$ is fixed or not by  $\bar{h}_1\cdots \bar{h}_n$ only depends on $x$ through $x|_n$. By the construction of $X_\omega$, the graph $x|_n$ isomorphic to a labeled  subgraph of $\tilde{\Gamma}_\omega$, hence of $\Gamma_\omega$. Let $\gamma\in\Gamma_\omega$ be the midpoint of this subgraph. It follows from the definition of $\ba,\bb, \bc, \bd$ that $h_1\cdots h_n\gamma\neq \gamma$. We conclude that $h_1\cdots h_n\neq e$ in $G_\omega$. 

The verification that $\iota$ is injective is similar. Let $h_1,\cdots, h_n\in \{a, \ob, \oc, \od\}$ be such that $h_1\cdots h_n\neq e$ in $G_\omega$. Then there is a vertex of the rooted tree, say $v\in\{0, 1\}^*$, which is moved by $h_1\cdots h_n$. Recall that the orbit $G_\omega\rho$  consists exactly of all sequences $\gamma\in\{0,1\}^*$ that are cofinal with $\rho$ (Proposition \ref{P: orbit}). In particular, $v$ is a prefix of infinitely many such sequences, each of which is moved by $h_1\cdots h_n$. Thus we can find  $\gamma\in\Gamma_\omega$ lying at a distance greater than $n$ from $\rho$, and such that $h_1\cdots h_n\gamma\neq \gamma$. Since we have chosen  $\gamma$ at a distance greater than $n$ from $\rho$, $\gamma$ is the midpoint of a connected graph of length $2n$  in $\Gamma_\omega$. This graph is also a  subgraph of $\tilde{\Gamma}_\omega$.  By Remark \ref{R: every word appears} and by shift-invariance, there exists $x\in X_\omega$ such that $x|_n$ is isomorphic to this graph (recall that $x|_n$ denotes the finite subgraph of $x$ spanned by $[-n, n]$). Again by  construction of $\ba, \bb, \bc, \bd$, we have $\bar{h}_1\cdots \bar{h}_n x=\varphi_\omega^{\pm k}(x)$, where $k> 0$ is the distance by which $\gamma$ is displaced on $\Gamma_\omega$ by the action of $h_1\cdots h_n$. But Corollary \ref{L: minimal} implies that there is no periodic point for $\varphi_\omega$, thus $\bar{h}_1\cdots\bar{h}_n (x)=\varphi_\omega^{\pm k}(x)\neq x$, and so $\bar{h}_1\cdots\bar{h}_n\neq e$ in $[[\varphi_\omega]]$. We conclude that $\iota$ is injective. 
\end{proof}

\begin{remark}\label{R: original action}
The usual action of $G_\omega$ on the Cantor set identified with the boundary of the tree $\{0, 1\}^\infty$ does not belong to the topological full group of a minimal homeomorphism of $\{0, 1\}^\infty$. To see this, note that  $\ob, \oc, \od$ fix $\rho$ but act non-trivially on every neighbourhood of $\rho$. This can not happen in the topological full group of a minimal homeomorphism.
\end{remark}

\section{Embedding $G_\omega$ in a simple Liouville group}

It is a classical result, due to Hall \cite{Hall}, Gorjuskin \cite{Gorjuskin} and Schupp \cite{Schupp} that every countable group can be embedded in a finitely generated simple group. These constructions always yield ``large'' ambient groups (e.g.\ non-amenable) even if the starting  group is ``small''. In this section we deduce the following consequence for the groups $G_\omega$ from the construction presented in Section 3.
 \begin{thm}\label{P: embedding}
Every Grigorchuk group $G_\omega$ can be embedded in a finitely generated, simple group that has trivial Poisson-Furstenberg boundary for every symmetric, finitely supported probability measure on it  (in other words, it has \emph{Liouville property} and in particular it is amenable).
\end{thm}
The class of groups with the Liouville property contains the class of subexponentially growing groups (Avez \cite{Avez}) and it is contained in the class of amenable groups (this is due to Furstenberg, see  \cite[Theorem 4.2]{Kaimanovich-Vershik}).

Recall that the \emph{complexity} of  a subshift $(X, \varphi)$ over a finite alphabet $\A$ is the function $\rho_X:\N\to\N$ that associates to $n$  the number of finite words in the alphabet $\A$ that appear as subwords of sequences in $X$. The proof of Theorem \ref{P: embedding} Êrelies on the following result from \cite{simpleliouville}.
 \begin{thm}[Theorem 1.2 in \cite{simpleliouville}]\label{T: complexity liouville}
 Let $(X,\varphi)$ be a subshift without isolated periodic points whose complexity satisfies $\rho_X(n)=o(\frac{n^2}{\log^2n})$. Then every symmetric, finitely supported probability measure on $[[\varphi]]$ has trivial Poisson-Furstenberg boundary.
 \end{thm}
  The next lemma allows to modify the construction in Section 3 to get an embedding in the commutator subgroup.

\begin{lemma}\label{L: commutator}
Let $(X,\varphi)$ be a Cantor minimal system and let $H<[[\varphi]]$ be the subgroup generated by elements of order 2. Then 
\begin{itemize}
\item[(i)] there exists a minimal Cantor system $(Y, \psi)$ and an embedding of $H$ into $[[\psi]]'$; 
\item[(ii)] if moreover $(X, \varphi)$ is a subshift with complexity function $\rho_X$, the system $(Y, \psi)$ in (i) can be chosen to be a subshift  whose complexity $\rho_Y$ satisfies
\[\rho_Y(n)\leq 2\rho_X(\lceil n/2\rceil).\]
\end{itemize}
\end{lemma}
\begin{proof}
(i). Set $Y= X\times \{1,2\}$ and define $\psi:Y\to Y$ by
\[\left\{\begin{array}{l}\psi(x, 1)=(x, 2)\\
                                   \psi(x, 2)= (\varphi(x), 1).\end{array}\right.\]
The system $(Y, \psi)$ is clearly minimal as soon as $(X, \varphi)$ is. Observe that $X\times \{1\}$ and ${X}\times\{2\}$ are $\psi^2$-invariant and that the restrictions of $\psi^2$ to $X\times\{1\}$ and to ${X}\times\{2\}$ are conjugate to $(X, \varphi)$.   This naturally defines two commuting subgroups of $[[\psi]]$ isomorphic to $[[\varphi]]$ that have support $X\times\{1\}$ and $X\times \{2\}$ respectively, and are denoted $[[\varphi]]_1$ and $[[\varphi]]_2$.  Consider the diagonal injection $\delta :[[\varphi]]\to[[\varphi]]_1\times[[\varphi]]_2\subset[[\psi]]$.   If $g\in[[\varphi]]$ has order 2, its image $\delta(g)$ equals the commutator $[g_1:\tau]=g_1\tau g_1\tau$, where $g_1\in[[\varphi]]_1$ is the image of $g$ under the natural isomorphism $[[\varphi]]\to[[\varphi]]_1$, and $\tau\in[[\psi]]$ is the element that coincides with $\psi$ on $X\times\{1\}$  and with $\psi^{-1}$ on ${X}\times\{2\}$. It follows that $\delta(H)$ is contained in $[[\psi]]'$, proving part (i).

(ii). If $(X, \varphi)$ is a subshift over the finite alphabet $\A$, the above construction is readily seen to be equivalent to the following one. Consider the alphabet $\mathcal{B}=\A\sqcup \{z\}$ where $z$ is a letter not in $\A$, and define $Y\subset \mathcal{B}^\Z$ to be the subshift consisting of all sequences of the form $y=\cdots zx_{-2}z x_{-1}.zx_0zx_1\cdots$ or $y=\cdots zx_{-2}z x_{-1}z.x_0zx_1\cdots$ where $x=\cdots x_{-2}x_{-1}.x_0x_1\cdots$ is a sequence in $X$. There is a natural partition of   $Y$ in two sets homeomorphic to $X$ which are given by the parity of the position of appearances of $z$. These two sets play the roles of  $X\times\{1\}$ and ${X}\times\{2\}$ in the proof of part (i) and the same argument yields an embedding of $H$ in $[[\psi]]'$. The claim on the complexity is also clear from this description.
 \end{proof}
 Since the groups $G_\omega$ are generated by elements of order 2, we get the following improvement of Theorem \ref{main}.
 \begin{cor}\label{C: commutator}
 Every Grigorchuk group $G_\omega$ embeds in the commutator subgroup of the topological full group of a minimal subshift.
 \end{cor}
 This already gives an embedding of $G_\omega$ into a finitely generated, simple amenable group by the result of Juschenko and Monod \cite{Juschenko-Monod:simpleamenable}. To prove that  the Liouville property holds, we estimate the complexity of  the subshift $(X_\omega, \varphi_\omega)$ constructed in Section 3, then apply part (ii) of Lemma \ref{L: commutator} and  Theorem \ref{T: complexity liouville}. The following lemma is enough to conclude the proof of Theorem \ref{P: embedding}.
 
 \begin{lemma} 
 For every $\omega\in \Omega$ the complexity of the subshift $(X_\omega, \varphi_\omega)$ constructed in Section 3 satisfies $\rho_{X_\omega}(n)\leq 6n$ for every $n \in \N$.
 \end{lemma}
 \begin{remark}
  Here we see $(X_\omega, \varphi_\omega)$ as a subshift in the usual sense, over the (formal) alphabet $\A=\{\Theta, \Lambda_\mathbf{0}, \Lambda_\mathbf{1}, \Lambda_\mathbf{2}\}$.  
 \end{remark}
 \begin{proof}
 Let $x\in X_\omega$ and consider a finite subword $w$ of $x$ of length $n\in\N$. Let $m\in \N$ be the smallest integer such that $n\leq 2^m$. By part (ii) of Lemma \ref{P: labelling} Êthe word $w$ appears as a subword of a word of the form $\tilde{\Gamma}_{m-1, \omega}*\Delta*\tilde{\Gamma}_{m-1, \omega}$, where $\Delta\in\{\Lambda_\mathbf{0}, \Lambda_\mathbf{1}, \Lambda_\mathbf{2}\}$. Such a word is uniquely determined by $\Delta$ (3 possibilities) and by the position of its first letter ($2^m$ possibilities, since we may assume without loss of generality that the first letter of $w$ is in  $\tilde{\Gamma}_{m-1, \omega}*\Delta$). Thus $\rho_{X_\omega}(n)\leq 3\cdot 2^m\leq 6n$, where we have used that $2^{m-1}< n$ by the choice of $m$.
 \end{proof}

 \appendix
In this Appendix we give proofs of the statements in Section \ref{S: Schreier}, namely Propositions \ref{P: orbit}, Proposition \ref{P: Schreier} and Lemma \ref{P: labelling}.

We first state a remark that follows from the definitions of $a, b_\omega,c_\omega, d_\omega$.
\begin{remarkappendix}\label{R: moves}
Given any binary string $\gamma\in \{0, 1\}^\infty$, a generator 
$s\in \{a, b_\omega, c_\omega, d_\omega\}$ that acts non-trivially on $\gamma$ can act in two possible ways:
\begin{itemize}
\item[(M1)] flip the first digit of $\gamma$ (if and only if $s=a$);
\item[(M2)] flip the digit after the first appearance of $0$ in $\gamma$ (only if $s\in \{b_\omega, c_\omega, d_\omega\}$).
\end{itemize}
Moreover for every $\gamma\neq\rho$ it is possible to choose $s\in \{ b_\omega, c_\omega, d_\omega\}$ that acts on $\gamma$ by a move of type (M2). 
\end{remarkappendix}

\begin{proof}[Proof of Proposition \ref{P: orbit}]
Since  the action of every generator of $G_\omega$ changes at most one letter of each binary string,  every string in the orbit of $\rho$ is cofinal with $\rho$. Conversely, it is easy to see that starting from $\rho$ and performing only moves of type (M1) and (M2) it is possible to reach every string which is cofinal with $\rho$. \qedhere
\end{proof}
\begin{proof}[Proof of Proposition \ref{P: Schreier}] 
Since $\rho$ is fixed by $b_\omega, c_\omega, d_\omega$, it has three loops and it is the endpoint of a simple edge (given by the action of $a$).
If $\gamma\neq \rho$ there exists exactly one generator  $s\in \{b_\omega, c_\omega, d_\omega\}$,  such that $s\gamma=\gamma$. The generator $s$  is determined by the value $\omega(m)$, where $m$ is the position of the first 0 appearing in $\gamma$. The two other generators $s', s''\in \{b_\omega, c_\omega, d_\omega\}$ act non-trivially on $\gamma$ by a move of type (M2), in particular $s'\gamma=s''\gamma$. The generator $a$ acts non-trivially on $\gamma$ by a move of type (M1), and thus 
$a\gamma\neq s'\gamma$. Hence the vertex corresponding to $\gamma$ in $\Gamma_\omega$  is the endpoint of a loop, a double edge and a simple edge. Finally, observe that the only connected unlabelled graph respecting these local rules is the graph shown in Figure \ref{F: blocks}.\qedhere
\end{proof}

Before giving the proof of Lemma \ref{P: labelling} we recall some terminology. The \emph{Gray code} is a classical non-standard way to code natural numbers with  binary strings, having the property that two consecutive strings differ only by one bit. It is a well-known fact that the Gray code can be used to describe the orbital Schreier graphs of several groups acting on rooted trees, including all groups $G_\omega$ (cf. Bartholdi, Grigorchuk and \v{S}uni\'c \cite[Section 10.3]{Bartholdi-Grigorchuk-Sunic:branchgroups} where this is mentioned explicitly for the first Grigorchuk group).  We recall the construction of the Gray code in the definition below; the reader should be warned that the roles of 0 and 1 are exchanged here with respect to the more usual definition.
 \begin{definappendix}  
 
 The \emph{Gray code order} is the numeration of the set $\{0, 1\}^l$ of binary sequences of length $l$, denoted $(r^{(l)}_j)_{0\leq j\leq 2^{l}-1}$, defined recursively as follows. Set $r^{(1)}_0=1$ and $r^{(1)}_1=0$. For every $l\geq 0$ and $0\leq j\leq 2^l-1$ set 
\[
r_j^{(l+1)}=r_j^{(l)}1;\qquad
r_{2^l+j}^{(l+1)}=r_{2^l-1-j}^{(l)}0.\]
 In plain words, the first $2^l$ terms of the sequence $(r^{(l+1)}_j)_{0\leq j\leq 2^{l+1}-1}$ are obtained by attaching the digit 1 to the sequence $(r^{(l)}_j)_{0\leq j\leq 2^{l}-1}$, and the last $2^l$ terms are obtained by  reversing the order of the sequence $(r^{(l)}_j)_{0\leq j\leq 2^{l}-1}$ and attaching the digit $0$ at the end. 
\end{definappendix}
For example, the Gray code order of binary strings of length 4 is $1111;\: 0111;$ $\: 0011;\:$ $ 1011;\: 1001;\:$ $ 0001;\: 0101;$ $\: 1101;\: 1100;\: 0100;\: 0000;\: 1000;\: 1010;\: 0010;\: 0110;\: 1110.$

 \begin{definappendix} The \emph{Gray code enumeration} of the orbit $G_\omega\rho$ is the sequence $(\rho_j)_{j\in\N}$, taking values in $\{0, 1\}^\infty$, where $\rho_j$ is obtained by attaching infinitely many digits 1 to  $r^{(l)}_j$, for any $l>0$  such that $r^{(l)}_j$ is defined (i.e. such that $j\leq 2^l-1$). Note that, for any such $l$, the string $r^{(l+1)}_j$ is obtained from $r^{(l)}_j$ by attaching the digit 1 on the right, in particular $\rho_j$ does not depend on the choice of $l$.
 \end{definappendix}

\begin{lemmaappendix}\label{L: sequence vertices}
The sequence $(\rho_j)$ coincides with the sequence of vertices of $\Gamma_\omega$, ordered by increasing distance from $\rho$. 
\end{lemmaappendix}
\begin{proof}
It follows from the definitions that $\rho_0=\rho$, and that $\rho_{j+1}$ is obtained from $\rho_j$ by performing a move of type (M1) if $j$ is even and a move of type (M2) if $j$ is odd (this can be immediately proven by induction on $l$ from the definition of the Gray code order). Thus the sequence $(\rho_j)$ coincides with the sequence of vertices of $\Gamma_\omega$ by Remark \ref{R: moves}.
\end{proof}
\begin{proof}[Proof of Lemma \ref{P: labelling}]
For $j\geq 1$ let $m_j$ be the position of the first 0 digit in $\rho_j$. There is exactly one generator $s\in\{b_\omega, c_\omega, d_\omega\}$ such that $s\rho_j=\rho_j$ and this generator is determined by $\omega(m_j)$ (more precisely $s=b_\omega, c_\omega, d_\omega$ if $\omega(m_j)=\mathbf{2, 1, 0}$, respectively). By comparing this observation with Definition \ref{D: labelling} and with Figure \ref{F: blocks}, we see that $\Delta_{i, \omega}=\Lambda_{\omega(m_{2i-1})}$. Set  $a_i=m_{2i-1}$. The sequence $(a_i)$ only depends on the definition of the Gray code order, and it is characterized by the recursion
\[\left\{\begin{array}{lr} a_{2^{n}}=n+1 &\text{for every }n\geq 0,\\
a_{2^{n}+j}=a_{2^{n}-j}& \text{for every }1\leq j\leq 2^{n}-1.
\end{array}\right.
\]

To prove part (i) of Lemma \ref{P: labelling}, proceed by induction on $n$. First note that $a_1=1$, so that $\Delta_{1, \omega}=\Lambda_{\omega(1)}$ and $\tilde{\Gamma}_{1,\omega}=\Theta*\Lambda_{\omega(1)}*\Theta$.  Suppose that the first $2^{n+1}$ vertices of $\tilde{\Gamma}_\omega$ span the graph $\tilde{\Gamma}_{n, \omega}$. Then, using the recursion for $a_j$, we have
\[ \tilde{\Gamma}_{n+1, \omega}=\Theta*\Lambda_{\omega(a_1)}*\Theta*\cdots *\Lambda_{\omega(a_{2^{n}})}*\cdots*\Theta*\Lambda_{\omega(a_{2^{n+1}-1})}*\Theta=\tilde{\Gamma}_{n, \omega}*\Lambda_{\omega(n+1)}*\tilde{\Gamma}_{n, \omega}^R,\]
where $\tilde{\Gamma}_{n, \omega}^R$ denotes $\tilde{\Gamma}_{n, \omega}$ reversed from right to left. However we may assume by induction  that $\tilde{\Gamma}_{n, \omega}$ is symmetric with respect to this operation, thus $\tilde{\Gamma}_{n+1, \omega}=\tilde{\Gamma}_{n, \omega}*\Lambda_{\omega(n+1)}*\tilde{\Gamma}_{n, \omega}$ concluding the proof of part (i).

To prove part (ii), fix $n$ and use part (i) to prove by induction on $m\geq 1$ that $\tilde{\Gamma}_{n+m, \omega}$ has the form
\[\tilde{\Gamma}_{n+m, \omega}=\tilde{\Gamma}_{n, \omega}*\Delta_{1, \sigma^n\omega}*\tilde{\Gamma}_{n, \omega}*\Delta_{2, \sigma^n\omega}\cdots\tilde{\Gamma}_{n, \omega}*\Delta_{2^{m+1}-1, \sigma^n\omega}*\tilde{\Gamma}_{n, \omega},\]
then let $m\to\infty$ to  conclude the proof.\qedhere

\end{proof}

  \bibliographystyle{alpha}
  \bibliography{these5-8-14}

\end{document}